\documentclass[12pt]{article}
\usepackage{amssymb,amsmath,amsfonts}
\usepackage{stmaryrd}
\usepackage{color}
\usepackage[latin1]{inputenc}

\usepackage{multirow}
\usepackage{caption}
\usepackage{subcaption}
\usepackage{enumerate}
\usepackage{epsfig}
\usepackage{epic}
\usepackage[font=small,skip=0pt]{caption}

\def\BBox{\kern  -0.2cm\hbox{\vrule width 0.2cm height 0.2cm}}

\newcommand{\hf}{\hspace*{0mm}\hspace{\fill}\hspace*{0mm}}

\newtheorem{theorem}{Theorem}[section]
\newtheorem{lemma}[theorem]{Lemma}

\newtheorem{corollary}[theorem]{Corollary}
\newtheorem{proposition}[theorem]{Proposition}

\newtheorem{observation}[theorem]{Observation}

\newcommand\size[1] {\left|{#1}\right|}

\newcommand\Setx[1] {\left\{{#1}\right\}}
\newcommand\patt[5] {\text{\small #1\ #2\ #3\ #4\ #5}}
\newcommand\alphabet {\{\text{\small A,\,B,\,C,\,D}\}}

\title{\textbf{Treelike snarks}}

\author{Mari\'{e}n Abreu$^{1}$\thanks{The research that led to the present paper was partially supported by a grant of the group GNSAGA of INdAM and by the Italian Ministry Research Project PRIN 2012 ``Geometric Structures, Combinatorics and their Applications''.}
, Tom\'{a}\v{s} Kaiser$^{2}$\thanks{Supported by project GA14-19503S of the Czech Science Foundation.
\newline {\em
Email addresses:} marien.abreu@unibas.it (M. Abreu),~
kaisert@kma.zcu.cz (T. Kaiser), ~ domenico.labbate@unibas.it (D. Labbate), ~ giuseppe.mazzuoccolo@univr.it (G. Mazzuoccolo)}
,\\
Domenico Labbate$^{1}$$^*$, Giuseppe Mazzuoccolo$^{3}$$^*$
\\[2ex]
\footnotesize $^1$ Dipartimento di Matematica, Universit\`{a} degli Studi della Basilicata,\\[-3mm]
\footnotesize Viale dell'Ateneo Lucano, I-85100 Potenza, Italy.\\
\footnotesize $^2$ Department of Mathematics, Institute for Theoretical Computer Science\\[-3mm]
\footnotesize and the European Centre of Excellence NTIS (New Technologies for the Information Society),\\[-3mm]
\footnotesize University of West Bohemia, Univerzitn\'{\i} 8, 306 14 Plze\v{n}, Czech Republic.\\
\footnotesize $^3$ Dipartimento di Informatica,\\[-3mm]
\footnotesize Universit\`{a} degli Studi di Verona, Strada le Grazie 15, 37134 Verona, Italy.}

\date{}

\begin{document}
\maketitle
\begin{abstract}
  We study snarks whose edges cannot be covered by fewer than five
  perfect matchings. Esperet and Mazzuoccolo found an infinite family
  of such snarks, generalising an example provided by H\"{a}gglund. We
  construct another infinite family, arising from a generalisation in
  a different direction. The proof that this family has the requested
  property is computer-assisted. In addition, we prove that the snarks
  from this family (we call them \emph{treelike snarks}) have circular
  flow number $\phi_C (G)\ge5$ and admit a 5-cycle double cover.

\end{abstract}



\section{Introduction}\label{intro}

All graphs considered are finite and simple (without loops or multiple edges). We shall use the term multigraph when multiple
edges are permitted.
Most of our terminology is standard; for further definitions and notation not explicitly stated in the paper, please refer to \cite{BM}.

By Vizing's theorem, the edge chromatic number of every cubic
graph is either three or four. In order to study the class of cubic graphs with edge chromatic number equal to four, it is usual to exclude certain trivial modifications. Thus, a {\em snark} (cf. e.g.~\cite{HS}) is defined as a bridgeless cubic graph with edge chromatic
number equal to four that contains no circuits of length at most four and no non-trivial $3$-edge cuts.

There is a vast literature on snarks and their properties --- see, e.g., \cite{ALRS,BGHM,KaRa10,Ko96-1,MaRa06}. The interested reader will find an introduction to the field in, e.g., \cite{HS} or \cite{Zh97}.

A \emph{perfect matching}, or {\em 1-factor}, in a graph $G$ is a regular spanning subgraph of degree $1$.
In this context, a {\em cover} of $G$ is a set of perfect matchings of $G$ such that each edge of $G$ belongs to at
least one of the perfect matchings.
Following the terminology introduced in \cite{BC}, the {\em excessive index of $G$}, denoted by
$\chi'_e(G)$, is the least integer $k$ such that the edge-set of $G$ can be covered by $k$ perfect matchings. Note that the excessive index is sometimes also called {\em perfect matching index} (see \cite{FV}).

The main source of motivation for the above notion is the conjecture
of Berge which asserts that the excessive index of any cubic
bridgeless graph is at most 5. As proved recently by the third
author~\cite{M}, this conjecture is equivalent to the famous
conjecture of Berge and Fulkerson~\cite{Fu71} that the edge-set of
every bridgeless cubic graph can be covered by six perfect matchings,
such that each edge is covered precisely twice.

It is NP-complete to decide for a cubic bridgeless graph $G$ whether $\chi'_e(G)=3$, since this property is equivalent to
3-edge-colourability. Similarly, Esperet and Mazzuoccolo~\cite{EM14}
proved that it is NP-complete to decide whether $\chi'_e(G)\leq 4$, as well as to
decide whether $\chi'_e(G)=4$.

As for cubic bridgeless graphs $G$ with $\chi'_e(G) \geq 5$, it was asked by Fouquet and Vanherpe~\cite{FV} whether the Petersen graph was the only such graph that is cyclically 4-edge-connected. H\"{a}gglund~\cite{Ha12} constructed another example (of order $34$) and asked for a characterisation of such graphs~\cite[Problem 3]{Ha12}. Esperet and Mazzuoccolo~\cite{EM14} generalized H\"{a}gglund's example to an infinite family.

We now outline the structure of the present paper, referring to the
later sections for the necessary definitions. In Section~\ref{tls}, we
construct a family of graphs called {\em treelike snarks}. We prove
that they are indeed snarks and (in Section~\ref{eitls}) that their
excessive index is greater than or equal to five. We thus expand the
known family of snarks of excessive index $\ge 5$, with a different
generalization of H\"{a}gglund's example than the one found by Esperet
and Mazzuoccolo in \cite{EM14}. The proof relies on the use of a
computer to determine a certain set of patterns (see Sections
\ref{sec:patterns} and \ref{appendix}).

In Section~\ref{cfntls}, we recall the definition of the circular flow
number and the $5$-Flow Conjecture of Tutte. We show that treelike snarks are, in a sense, also critical for this conjecture --- namely, their circular flow number is greater than or equal to five.

Section~\ref{cdctls} is devoted to cycle double covers. Since it is known that any cubic graph $G$ that is a counterexample to the Cycle Double Cover Conjecture satisfies $\chi'_e(G)\geq 5$, it is natural to ask whether treelike snarks admit cycle double covers. We show that this is indeed the case. In fact, using a new general sufficient condition for the existence of a 5-cycle double cover, we show that treelike snarks satisfy the $5$-Cycle Double Cover Conjecture of Preissmann \cite{Preis81} and Celmins \cite{Ce84}.

\section{Preliminaries}\label{prelim}

For a given graph $G$, the vertex set of $G$ is denoted by $V(G)$, and its edge set by $E(G)$. Each edge is viewed as composed of two \emph{half-edges} (that are \emph{associated} to each other) and we let $E(v)$ denote the set of half-edges incident with a vertex $v$.

A \emph{join} in a graph $G$ is a set $J \subseteq E(G)$
such that the degree of every vertex in $G$ has the same parity as its
degree in the graph $(V(G),J)$. In the literature, the terms
\emph{postman join} or \emph{parity subgraph} have essentially the
same meaning. Throughout this paper, we will be dealing with cubic
graphs, so \emph{joins} will be spanning subgraphs where each vertex
has degree $1$ or $3$.

As usual, e.g., in the theory of nowhere-zero flows, we define a \emph{cycle} in a graph $G$ to be any subgraph $H \subseteq G$ such that each vertex of $H$ has even degree in $H$. Thus, a cycle need not be connected. A \emph{circuit} is a  connected $2$-regular graph. In a cubic graph, a cycle is a disjoint union of circuits and isolated vertices.

\begin{observation}\label{cyclevsjoin}
A subgraph $H \subseteq G$ is a cycle in $G$ if and only if $E(G) - E(H)$ is a join.
\end{observation}

An \emph{edge-cut} (or just \emph{cut}) in $G$ is a set $T \subseteq E(G)$
such that $G-T$ has more components than $G$, and $T$ is inclusionwise
minimal with this property. A cut is \emph{trivial} if it consists of
all edges incident with a particular vertex. A \emph{bridge} is a cut of size $1$. A graph is \emph{bridgeless} if it contains no bridge (note that with this definition, a bridgeless graph may be disconnected). A graph $G$ is said to be \emph{$k$-edge-connected} (where $k \ge 1$) if $G$ is connected and contains no cut of size at most $k-1$, i.e. $G$ is such that its edge-connectivity $k'(G) \ge k$. A set $S$ of edges of a graph $G$ is a \emph{cyclic edge cut} if $G-S$ has two components each of which contains
a cycle. We say that a graph $G$ is \emph{cyclically
  $k$-edge-connected} if $\size{E(G)} > k$ and each cyclic edge cut of $G$ has size at least $k$.

A \emph{cover} (or \emph{covering}) of a graph $G$ is a family $\cal
F$ of subgraphs of $G$, not necessarily edge-disjoint, such that
$\bigcup_{F \in {\cal F}} E(F) = E(G)$. A \emph{$(1,2)$-cover} is a
cover in which each edge appears at most twice.

Recall that the following lemma is very useful when
studying edge-colourability of cubic graphs (see, for instance, Lemma B.1.3 in \cite{Zh12}).

\begin{lemma}[Parity Lemma]\label{paritylemma}
Let $G$ be a cubic graph and let $c:E(G) \to \{1,2,3\}$ be a $3$-edge-colouring of
$G$. Then, for every edge-cut $T$ in $G$,
$$|T \cap c^{-1}(i)| \equiv |T| \pmod 2$$ for $ i \in \{1,2,3\}$.
\end{lemma}

\section{Treelike snarks}\label{tls}

As noted in Section~\ref{prelim}, we view each edge of a graph as
composed of two half-edges. We now extend the notion of a graph by
allowing for \emph{loose} half-edges that do not form part of any
edge; the resulting structures will be called \emph{generalised
  graphs}. A generalised graph is \emph{cubic} if each vertex is
incident with three half-edges.

We define a \emph{fragment} $F$ as a generalised cubic graph with
exactly five loose half-edges, ordered in a sequence (see Figure \ref{fig:fragment}).
The \emph{Petersen fragment} $F_0$ is the fragment with loose
half-edges ($a_1,\dots,a_5$) obtained from the Petersen graph (see
Figure \ref{petfrag}) as follows:
\begin{itemize}
\item in the Petersen graph, remove a vertex $x$, keeping the
  half-edges $a_3,a_4,a_5$ incident with its neighbours $y,z,t$,
  respectively,
\item subdivide the two edges incident with $y$,
\item and add half edges $a_1$, $a_2$ to the new vertices of the subdivision (see
  Figures \ref{petersen}, \ref{petfrag}) in any one of the two ways.
\end{itemize}
This fragment $F_0$ will be particularly important throughout the paper.

\begin{figure}[h]
  \centering
  \includegraphics[height=4cm]{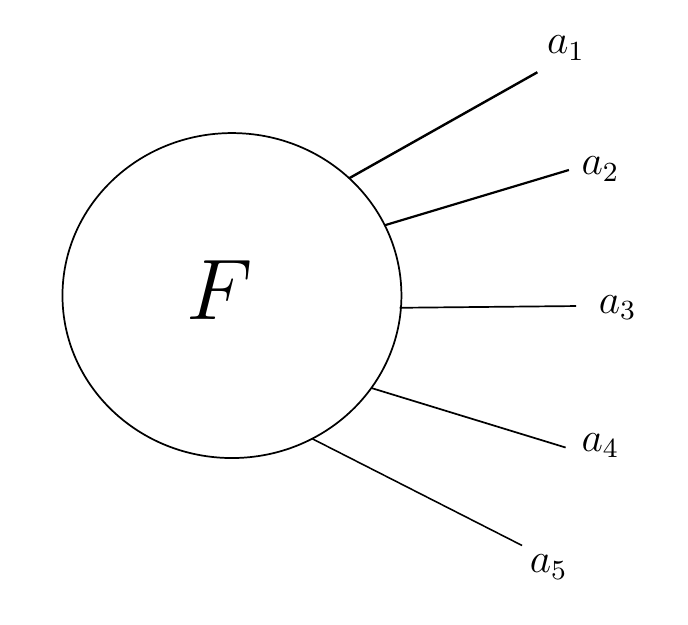}
  \caption{A fragment.}
  \label{fig:fragment}
\end{figure}

\begin{figure}[h]
  \hf
  \begin{subfigure}[t]{5cm}
		\centering
		\includegraphics[height=3.7cm]{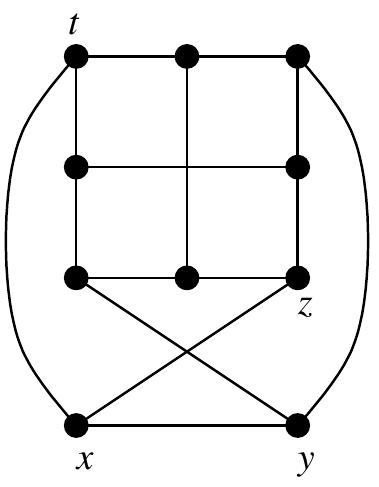}
		\caption{The Petersen graph.}\label{petersen}		
	\end{subfigure}
    \quad \quad \quad
	\begin{subfigure}[t]{5cm}
		\centering
		\includegraphics[height=3.7cm]{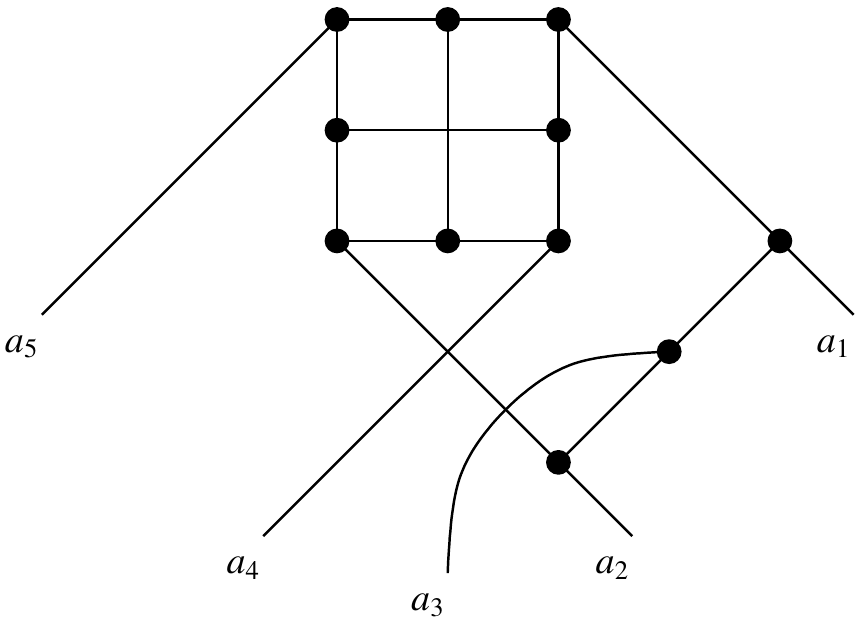}
		\caption{The Petersen fragment $F_0$.}\label{petfrag}
	\end{subfigure}
        \hf
	\caption{Constructing the Petersen fragment.}\label{fragments}
\end{figure}

A \emph{Halin graph} is a plane graph consisting of a planar representation of a tree without degree 2 vertices,
and a circuit on the set of its leaves (cf., e.g., \cite{Hal64,CNP83}).

Let $H_0$ be a cubic Halin graph consisting of the tree $T_0$ and the
circuit $C_0$.  The \emph{treelike snark} $G(T_0,C_0)$ is obtained by
the following procedure:
\begin{itemize}
\item for each leaf $\ell$ of $T_0$, we add a copy $F^\ell_0$ of the
  Petersen fragment $F_0$ with loose half-edges $(a_1,\dots,a_5)$ and
  attach the half-edge $a_3$ to $\ell$,
\item for each leaf $\ell$ of $T_0$ and its successor $\ell'$ with
  respect to a fixed direction of $C_0$, if $F^\ell_0$ has loose
  half-edges $(a_1,\dots,a_5)$ and $F^{\ell'}_0$ has loose half-edges
  $(a'_1,\dots,a'_5)$, then we join $a_4$ to $a'_2$ and $a_5$ to
  $a'_1$, obtaining new edges.
\end{itemize}
If there is no danger of a confusion, we abbreviate $G(T_0,C_0)$ to $G(T_0)$.

Some small examples of treelike snarks are shown in Figure~\ref{smalltls}.

\begin{figure}[h]	
	\centering
	\begin{subfigure}[h]{7cm}
		\centering
		\includegraphics[width=5cm]{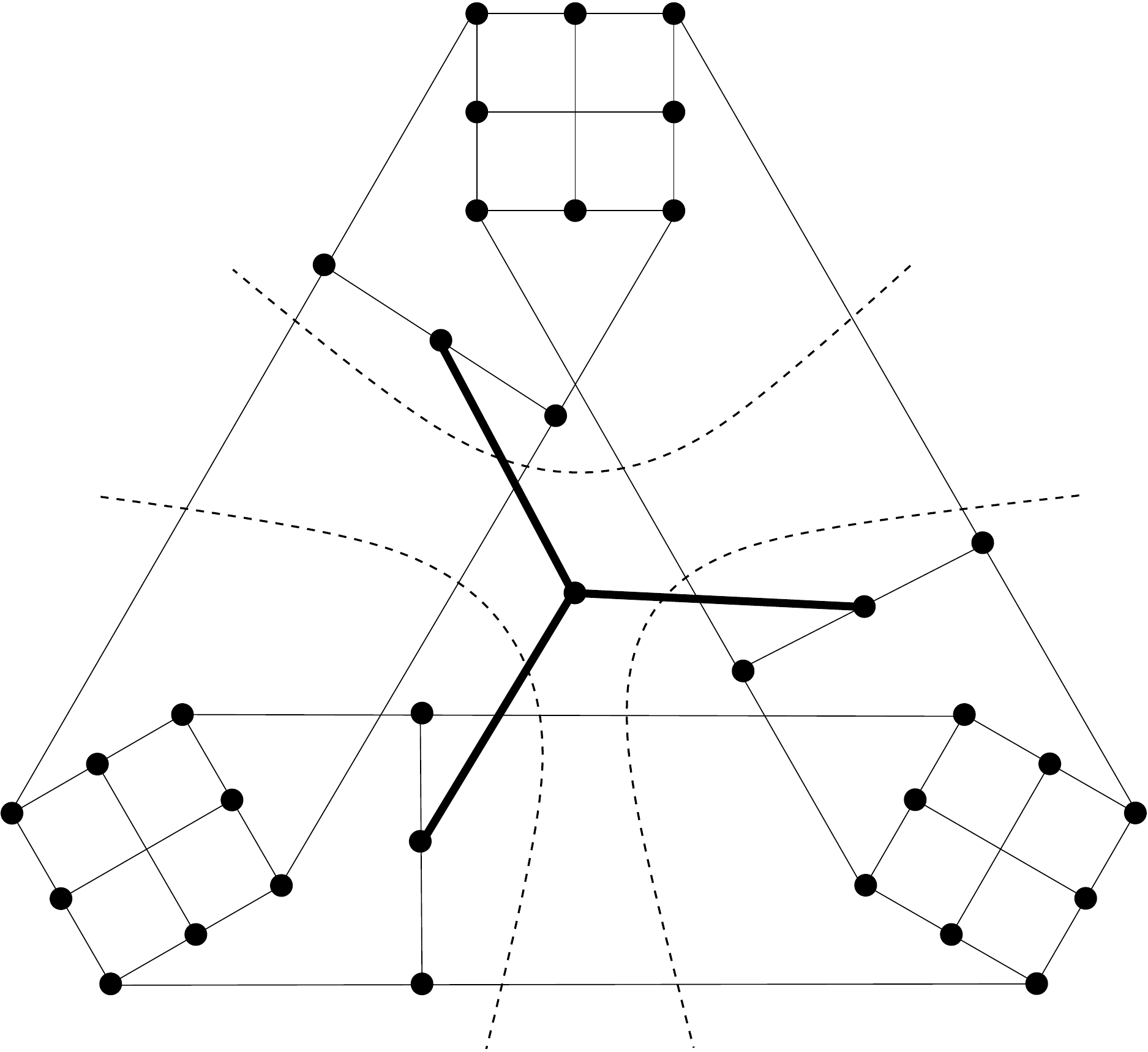}
		\caption{\small Windmill 1 \cite[Figure 7]{EM14}.}\label{windmill1}		
	\end{subfigure}
    \quad \quad \quad
	\begin{subfigure}[h]{7cm}
		\centering
		\includegraphics[width=5cm]{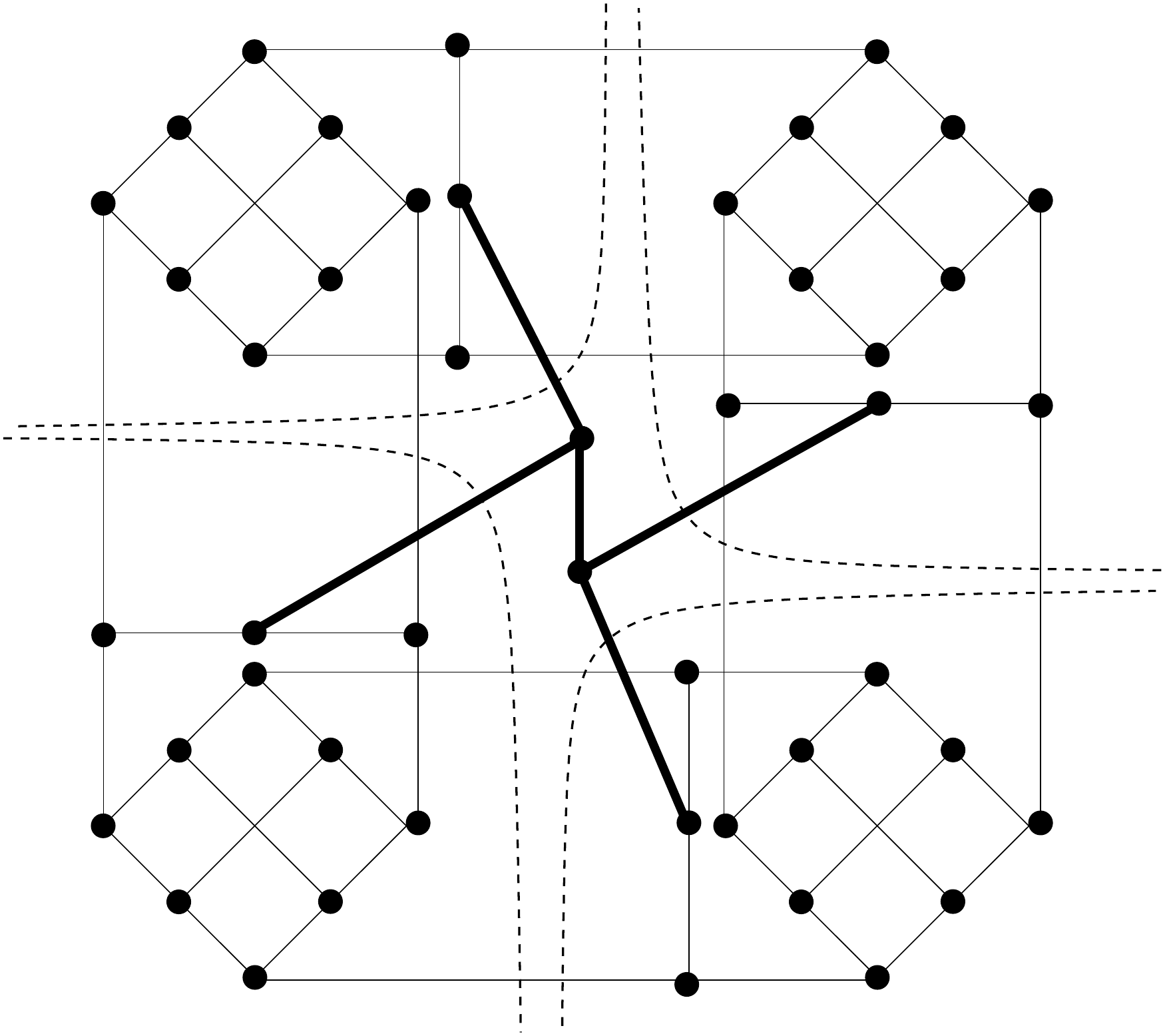}
		\caption{Blowup$(Prism,C_4)$ \cite[Figure 7]{Ha12}.}\label{hagglund}
	\end{subfigure}
	\caption{Small treelike snarks.}\label{smalltls}
\end{figure}

Let $F_1$ and $F_2$ be fragments with half edges $(a_1,a_2,a_3,a_4,a_5)$ and $(a'_1,a'_2,a'_3,a'_4,a'_5)$ respectively.
We define the \emph{sum} $F_1 + F_2$ of $F_1$ and $F_2$ as the fragment obtained from $F_1$ and $F_2$ by the following operations:
\begin{itemize}
\item joining the half-edges $a_5$ to $a'_1$ and $a_4$ to $a'_2$ into
  edges,
\item adding a new vertex adjacent to $a_3$, $a'_3$ and a new half-edge $a''_3$,
\end{itemize}
with half-edges $(a_1,a_2,a''_3,a'_4,a'_5)$ (see Figure \ref{fragsumfig}).

The \emph{fusion} of $F_1$ with $F_2$ is the cubic graph obtained from
the union of $F_1$ and $F_2$ by joining each pair of half-edges $a_i$,
$a'_{6-i}$, where $1\leq i \leq 5$ (see Figure \ref{fusion}).

Observe that any treelike snark $G(T_0,C_0)$ may be obtained as a
fusion of two finite sums of Petersen fragments (with appropriate
bracketings, determined by $T_0$). In fact, one of the sums can be
taken to be just a Petersen fragment.

\begin{figure}[h]	
	\centering
	\begin{subfigure}[h]{7cm}
		\centering
		\includegraphics[width=7cm]{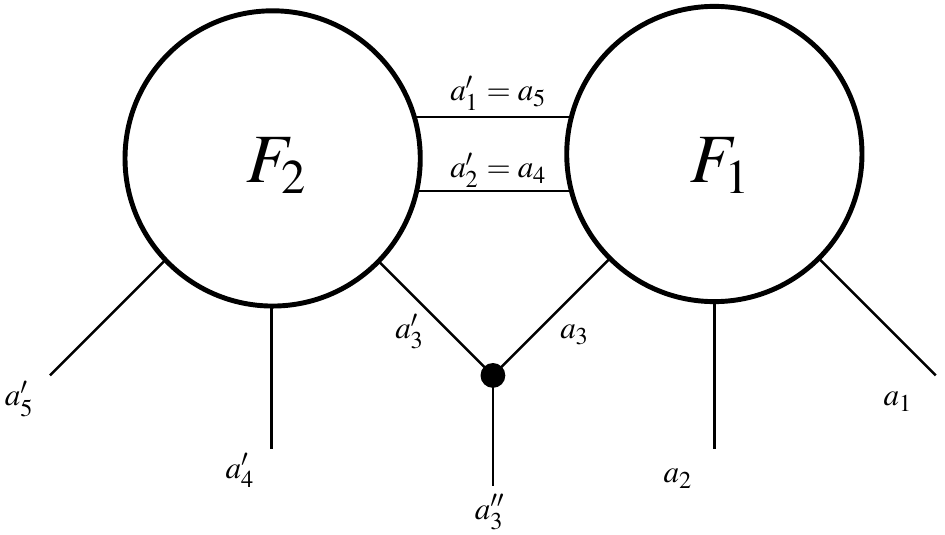}
		\caption{Fragment sum $F_1 + F_2$.}\label{fragsumfig}		
	\end{subfigure}
    \quad \quad \quad
	\begin{subfigure}[h]{7cm}
		\centering
		\includegraphics[width=6.5cm]{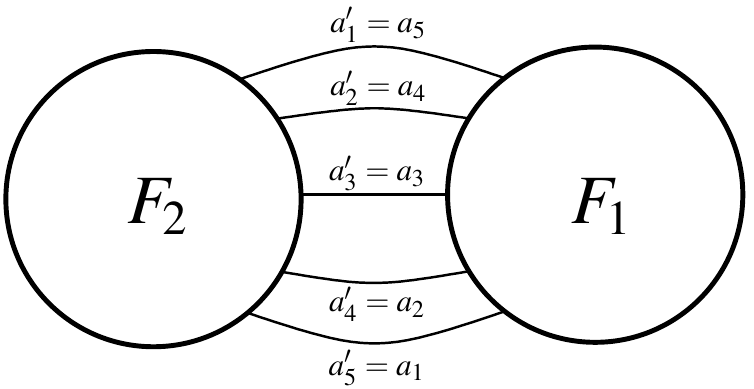}
		\caption{Fusion of $F_1$ and $F_2$.}\label{fusion}
	\end{subfigure}
	\caption{}
\end{figure}

The following properties of the Petersen fragment $F_0$, arising from its symmetries, will be useful later.

\begin{enumerate}[(i)]
\item The half-edges $a_1$ and $a_2$ are symmetrically equivalent, in
  the sense that reversing their order leads to an isomorphic fragment
  (with the notion of isomorphism defined in the natural way). We will
  call these half-edges the \emph{right spokes}.
\item The half-edges $a_4$ and $a_5$ are symmetrically equivalent. We
  will call these the \emph{left spokes}.
\item The endvertices of the half-edges $a_1,a_3$ and $a_2$ form a
  path of length two (in that order), which we will call the
  \emph{special path}. The half-edge $a_3$ may also be referred to as
  the \emph{central spoke}.
\end{enumerate}

The \emph{distance} between two half-edges is the length of a shortest
path connecting their (unique) vertices.

The following proposition justifies the phrase `treelike snark':
\begin{proposition}\label{tls-snark}
Treelike snarks are snarks.
\end{proposition}

\begin{proof}
  Let $G(T_0) = G(T_0,C_0)$ be a treelike snark. By construction,
  $G(T_0)$ is cubic.  As noted above, $G(T_0)$ is a fusion of a
  Petersen fragment with a finite sum of Petersen fragments
  $F^\ell_0$, where $\ell$ ranges over the leaves of $T_0$.

  We prove that $G(T_0)$ has girth 5. In each Petersen fragment $F_0$
  there is a $5$-circuit (a circuit of length $5$) but no shorter
  circuits. Suppose that $G(T_0)$ contains a circuit $C$ of length at
  most $4$ traversing some of the loose half-edges $(a_1,\dots,a_5)$
  of a Petersen fragment $F^\ell_0$. The distance from each left spoke
  $a_4,a_5$ to any other half-edge of $F_0$ is at least $3$, hence $C$
  traverses no left spoke of $F^\ell_0$. By the structure of $G(T_0)$
  as a sum of fragments, right spokes of $F^\ell_0$ are joined to left
  spokes of some other Petersen fragment, so no right spoke of
  $F^\ell_0$ is traversed by $C$ either. The only remaining loose
  half-edge of $F^\ell_0$ is $a_3$, a contradiction. Thus, $G(T_0)$ has
  girth $5$.

  We claim that $G(T_0)$ is cyclically $4$-edge-connected. Since
  $G(T_0)$ is clearly $3$-edge-connected, it suffices to show that
  there is no nontrivial $3$-edge-cut. This is clearly true of the
  Petersen fragment has no nontrivial $3$-edge-cut, and one can easily
  check that neither the sum nor the fusion with a Petersen fragment
  introduce a nontrivial $3$-edge-cut.
  
To determine the chromatic index of $G(T_0)$, suppose $G(T_0)$ admits
a 3-edge-colouring. Denote by $b_1$ and $b_2$ the two edges of a Petersen fragment $F^\ell_0$ which are incident with $a_1$ and $a_2$, respectively, and are not edges of the special path.  If $F^\ell_0$ has left spokes
of different colours, then by Lemma~\ref{paritylemma}, the same two colours also appear on $b_1$ and $b_2$, and it is easy to
derive a $3$-edge-colouring of the Petersen graph, a contradiction.
Therefore, we can assume that in each Petersen fragment, the two left spokes are given the same
colour. By construction, the right spokes of $F^\ell_0$ are left
spokes of another Petersen fragment, so they also share a
colour (in general, different from that of the left spokes). Suppose that a Petersen
fragment has the two left spokes of colour $a$ and the two right spokes
of colour $b$. Then $a=b$, since by Lemma \ref{paritylemma}, neither
$a$ nor $b$ can appear exactly twice in the $5$-cut separating the
fragment from the rest of the graph. In
addition, the lemma implies that all central spokes share the same
colour $a$, a contradiction since at least two of them are adjacent
edges. Thus, the chromatic index of $G(T_0)$ is $4$.
\end{proof}


\section{Patterns}
\label{sec:patterns}

The notion of a perfect matching is easily extended from graphs to
generalised graphs: it is simply a set $M$ of edges and loose half-edges
such that each vertex is incident with exactly one element of $M$. We
will be interested in $(1,2)$-covers of a given fragment by $4$ perfect
matchings. To describe the `behaviour' of a cover on the loose
half-edges $a_1,\dots,a_5$ of a fragment $F$, we introduce the
following definitions.

A \emph{pattern} $\pi$ is a sequence of five subsets of
$\alphabet$ of size $1$ or $2$ each, such that each symbol from
$\alphabet$ appears in an odd number of the subsets in
$\pi$. Examples of patterns are $\patt A A {AB}{AC}{AD}$ or
$\patt{AB}{AC}{AD}{BD}{BD}$ (we omit both the set brackets and the
parentheses enclosing a sequence).

Observe that any $(1,2)$-cover of a fragment by $4$ perfect
matchings determines a pattern in a natural way. For instance, if a
cover by perfect matchings $A,B,C,D$ is such that each of the loose
half-edges $a_1,\dots,a_5$ is contained in $A$, and in addition,
$a_3,a_4,a_5$ are contained in $B,C$ and $D$ respectively, then the
corresponding pattern is $\patt A A {AB}{AC}{AD}$.

The set of all patterns determined by $(1,2)$-covers of a fragment
$F$ by $4$ perfect matchings is called the \emph{pattern set of $F$}
and denoted by $\Pi(F)$.

Using a computer program implemented in C, we have determined the
pattern sets for each of the following fragments: $F_0$, $F_0+F_0$,
$(F_0+F_0)+F_0$, $F_0+(F_0+F_0)$, $(F_0+F_0)+(F_0+F_0)$. The results
are summarised in Section~\ref{appendix}. One obvious conclusion from
these data is that
\begin{equation*}
  \Pi(F_0+(F_0+F_0)) \neq \Pi((F_0+F_0)+F_0).
\end{equation*}
This somewhat surprising lack of associativity is illustrated in
Figure~\ref{f0nonaso}.

\begin{figure}[h]
	\centering
	\begin{subfigure}[t]{3cm}
		\centering
		\includegraphics[width=3cm]{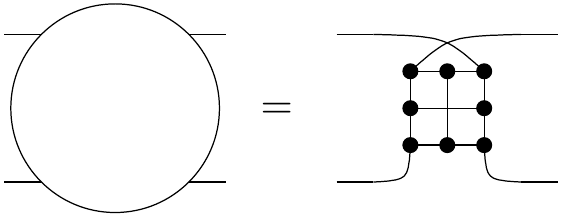}
	\end{subfigure}
	
\vspace{0.7cm}
	
	\begin{subfigure}[t]{7cm}
		\centering
		\includegraphics[width=7cm]{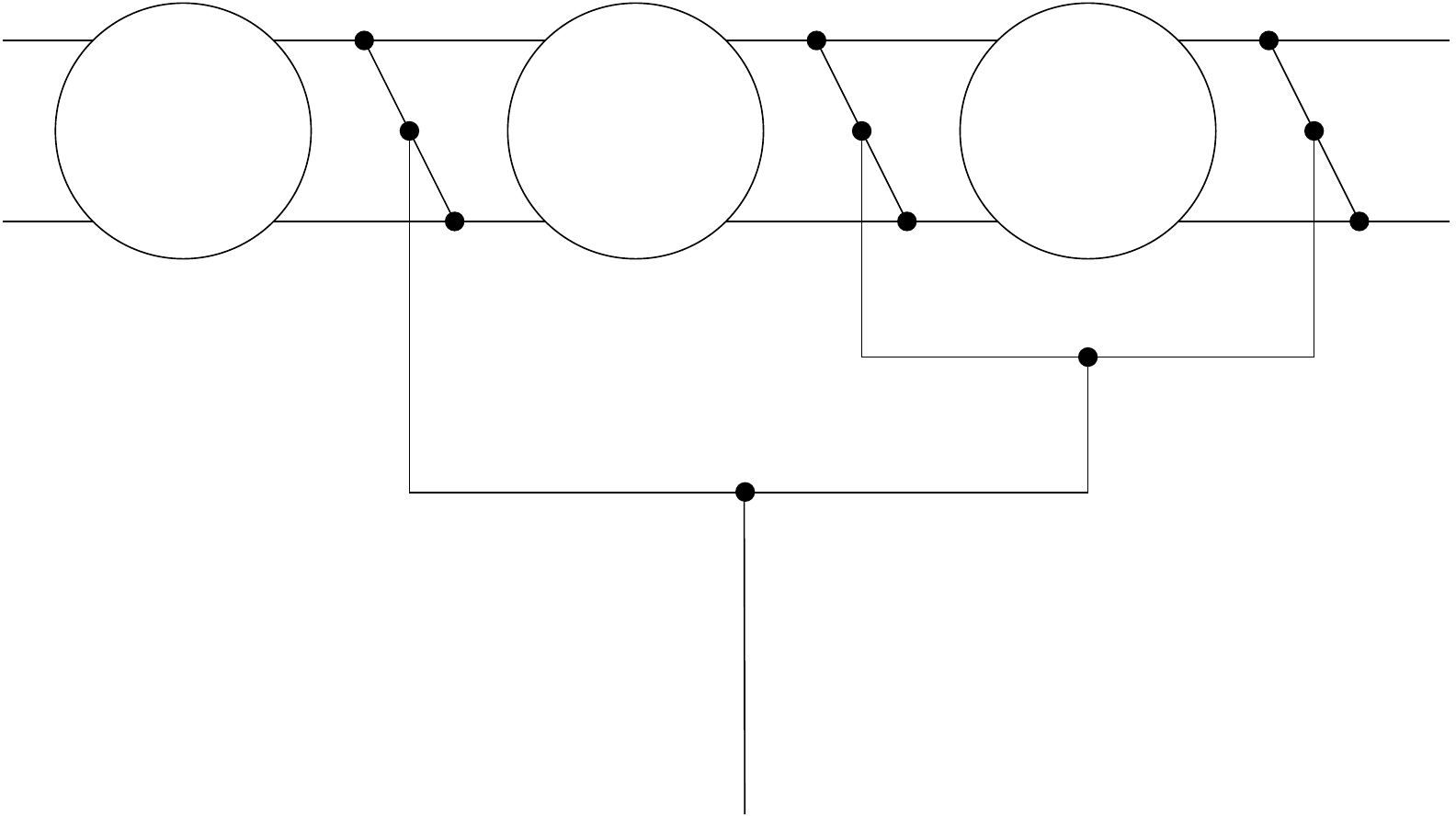}
		\caption{The fragment $F_0+(F_0+F_0)$.}\label{fxff}		
	\end{subfigure}
    \quad
	\begin{subfigure}[t]{7cm}
		\centering
		\includegraphics[width=7cm]{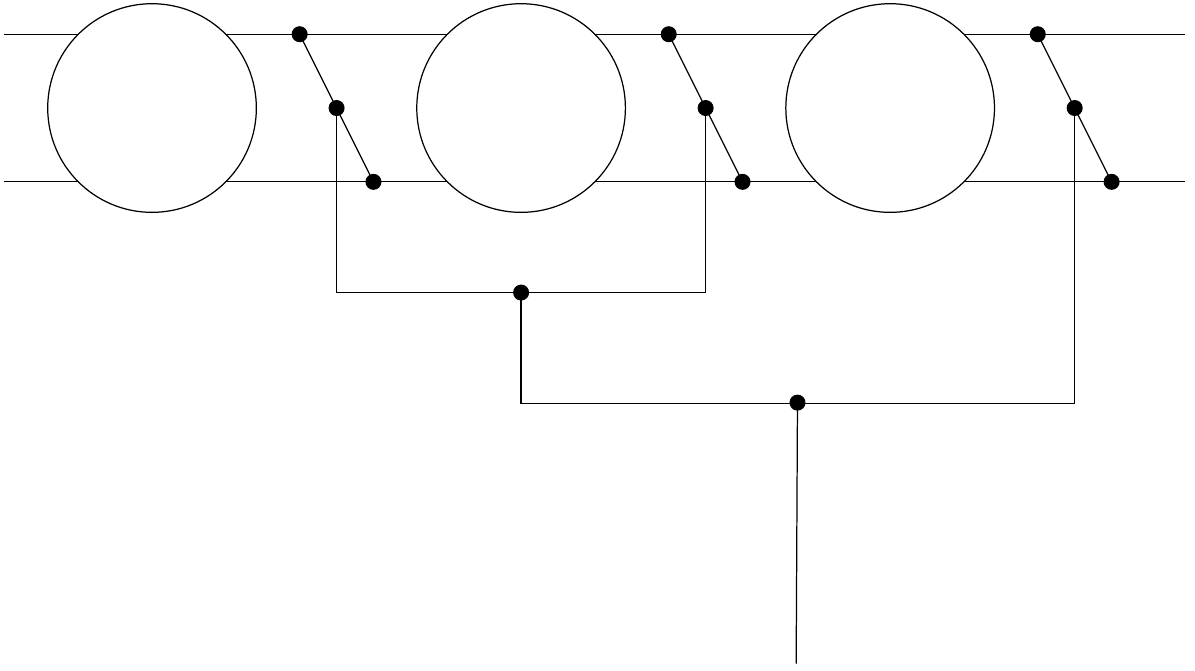}
		\caption{The fragment $(F_0+F_0)+F_0$.}\label{ffxf}
	\end{subfigure}
		\caption{Lack of associativity in sums of Petersen
                  fragments.}\label{f0nonaso}

\end{figure}

A further observation about the results presented in
Section~\ref{appendix} is given by the following lemma:

\begin{lemma}\label{patternsets}
Let $F_0$ be the Petersen fragment. The following inclusions hold:
\begin{enumerate}[(i)]
\item $\Pi(F_0+(F_0+F_0)) \subset \Pi(F_0+F_0)$,
\item $\Pi((F_0+F_0)+(F_0+F_0)) \subset \Pi((F_0+F_0)+F_0)$.
\end{enumerate}
\end{lemma}
\begin{proof}
By inspection of the lists in Section~\ref{appendix}.
\end{proof}

\section{Excessive index}\label{eitls}

\begin{theorem}\label{main1}
Treelike snarks have excessive index at least $5$.
\end{theorem}

\begin{proof}
Let $G(T_0)$ be a treelike snark with underlying tree $T_0$. We proceed by induction on the order of $T_0$.
If $T_0 = K_{1,3}$, then by \cite[Sec. 3]{Ha12} and \cite[Theorem 6]{EM14}, the treelike snark $G(T_0)$ has excessive index $5$.

Suppose the statement is true for all treelike snarks with $|V(T_0)| <
n$, where $n > 4$.

Suppose further that the underlying tree $T_0$ has $n$ vertices, and
let $U = u_1,u_2, \ldots, u_t$ be a longest path of vertices of degree
$3$ in $T_0$. Then the part of $T_0$ around the endvertex $u_1$ of $U$
coincides with one of the possibilities in Figure \ref{treereduc}.

\begin{figure}[h]
	\centering
	\begin{subfigure}[h]{3cm}
		\centering
		\includegraphics[scale=0.35]{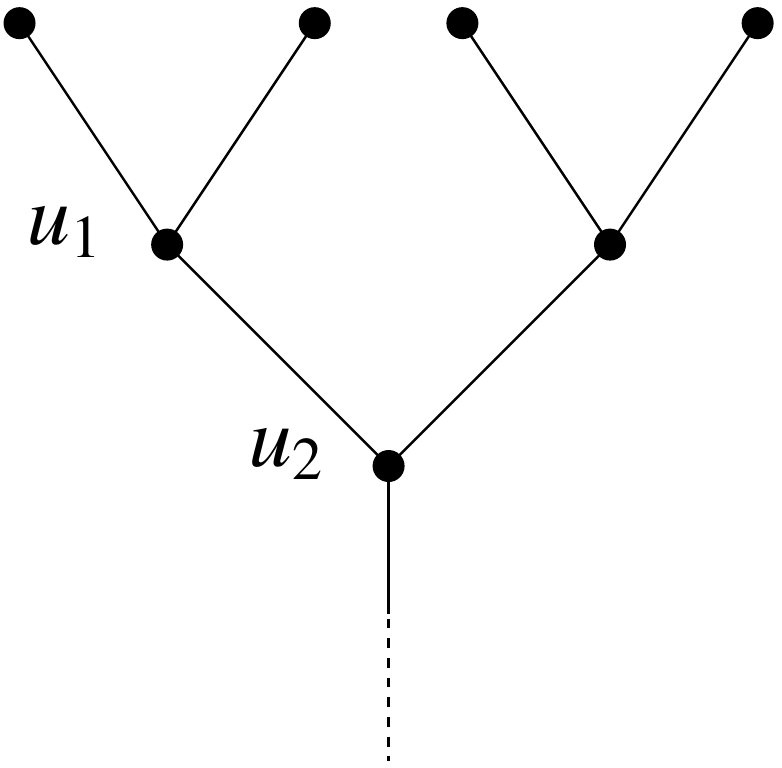}
		\caption{}\label{treereduc-a}		
	\end{subfigure}
    \quad
	\begin{subfigure}[h]{4cm}
		\centering
		\includegraphics[scale=0.35]{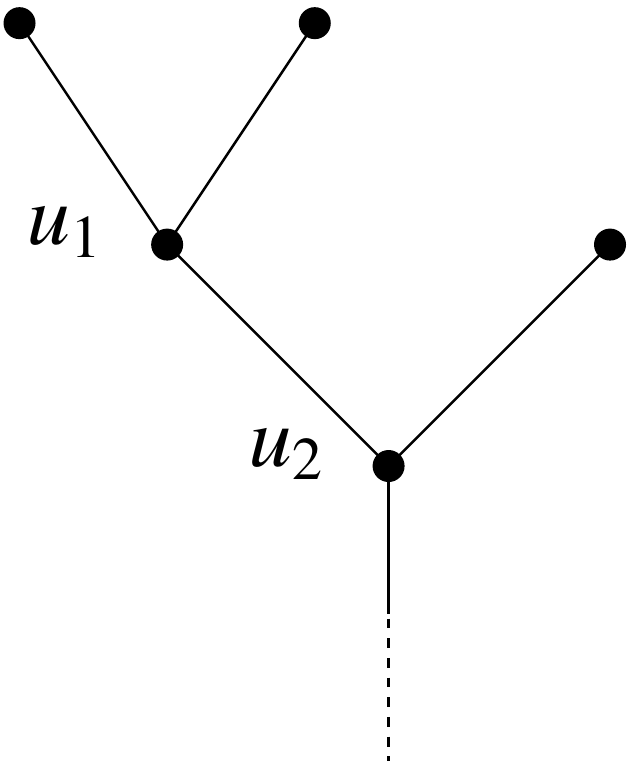}
		\caption{}\label{treereduc-b}
	\end{subfigure}
   \quad
	\begin{subfigure}[h]{4cm}
		\centering
		\includegraphics[scale=0.35]{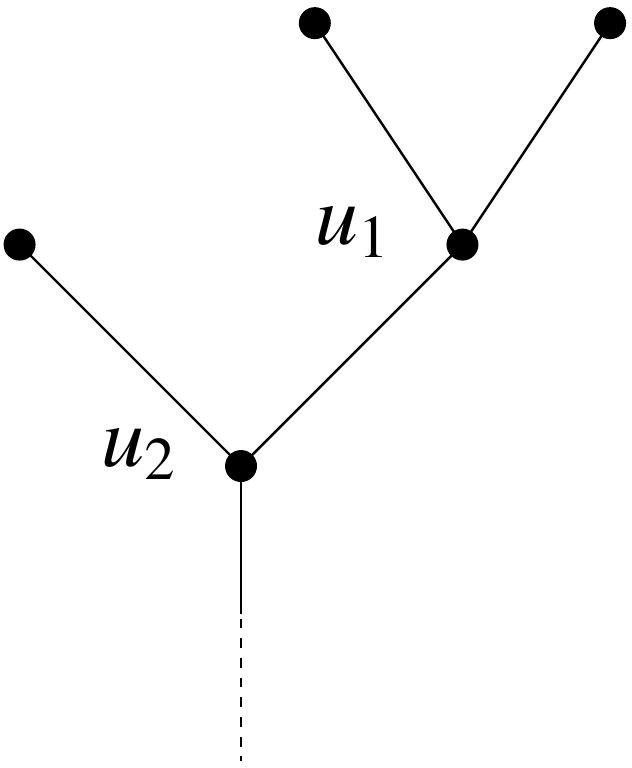}
		\caption{}\label{treereduc-c}
	\end{subfigure}
	\caption{Possible cases in the proof of
          Theorem~\ref{main1}. Any vertex represented as a leaf is a
          leaf of $T_0$.}\label{treereduc}
\end{figure}

We show that case (a) reduces to case (b). Replace the fragment
$(F_0+F_0)+(F_0+F_0)$ corresponding to the part of $T_0$ shown in
Figure~\ref{treereduc}(a) by $(F_0+F_0)+F_0$, obtaining a cubic graph
$G'$. Suppose that the excessive index of $G(T_0)$ is less than or
equal to $4$. Then it is, in fact, equal to $4$ (as we know that
$G(T_0)$ is a snark). Consider a $(1,2)$-cover of $G(T_0)$ by $4$
perfect matchings. By Lemma~\ref{patternsets}(ii), the pattern induced
on the loose half-edges of $(F_0+F_0)+(F_0+F_0)$ can be extended to a
$(1,2)$-cover by $4$ perfect matchings of $(F_0+F_0)+F_0$, and
hence to that of $G'$. This contradicts the induction hypothesis, as
the latter implies that the excessive index of $G'$ is greater than
$4$. Consequently, the excessive index of $G(T_0)$ is greater than
$4$.

\begin{figure}
	\centering
		\includegraphics[height=7cm]{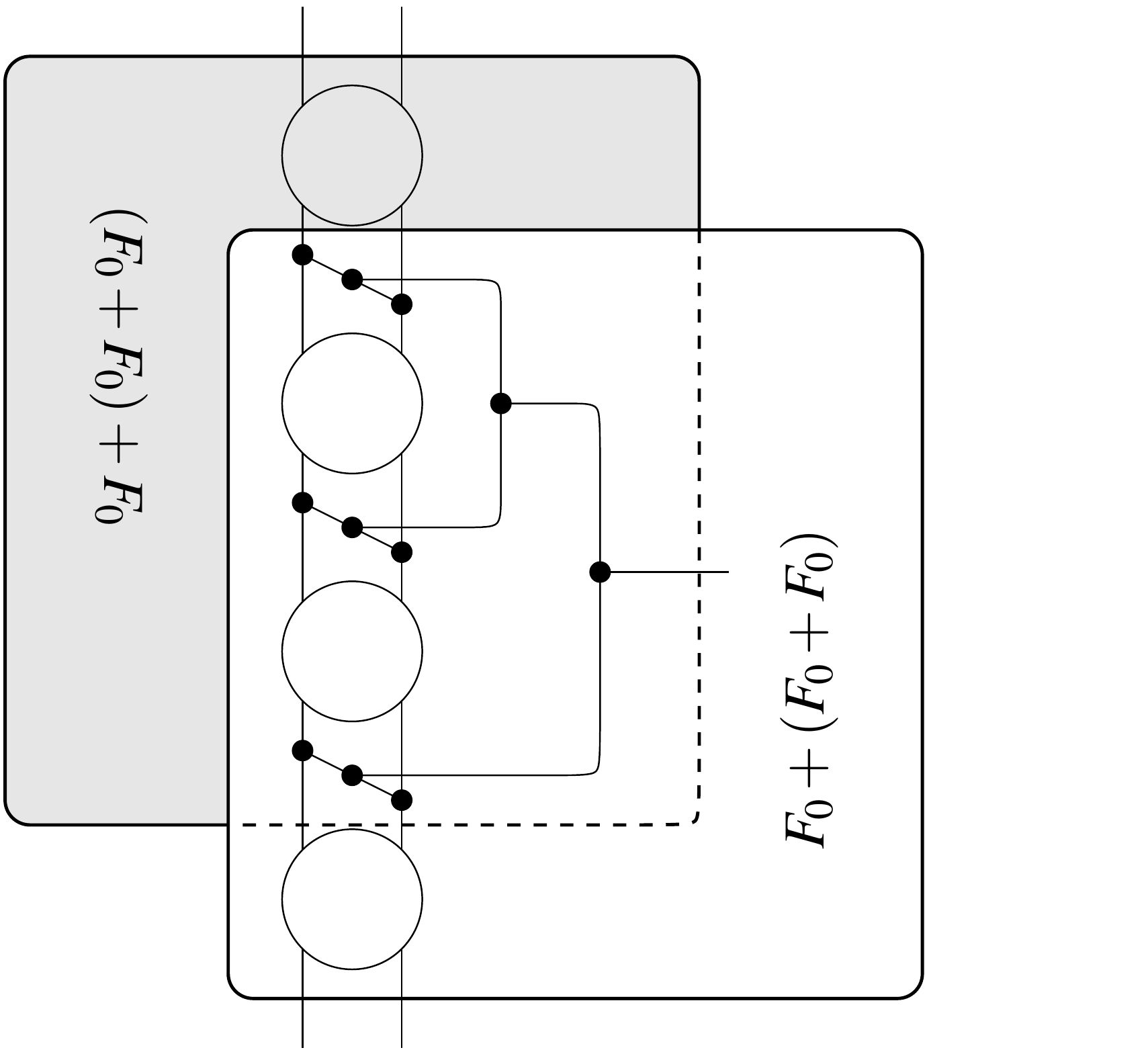}
		\caption{Moving from case (b) to case (c) in the proof
                of Theorem~\ref{main1}.}\label{shift}		
\end{figure}

In a similar way, we can reduce case (c) using
Lemma~\ref{patternsets}(i). Although the lemma implies no direct
reduction for case (b), Figure~\ref{shift} shows that by expressing
$G(T_0)$ as a sum (and fusion) of fragments in a different way, case
(b) is transformed into case (c), which is reduced as before. The
proof is thus complete.
\end{proof}

\section{Circular flow number}\label{cfntls}

We recall the notion of circular nowhere-zero $r$-flow, first
introduced in \cite{GTZ}. Let $G=(V,E)$ be a graph.

Given a real number $r\geq 2$, a {\em circular nowhere-zero $r$-flow}
($r$-CNZF for short) in $G$ is an assignment $f:\,E \rightarrow
[1,r-1]$ and an orientation $D$ of $G$, such that $f$ is a flow in
$D$. That is, for every vertex $x\in V$, $\sum_{e\in
  E^+(x)}f(e)=\sum_{e\in E^-(x)}f(e)$, where $E^+(x)$, respectively
$E^-(x)$, are the sets of edges directed from, respectively toward,
$x$ in $D$.

The {\em circular flow number} $\phi_c(G)$ of $G$ is the infimum of
the set of numbers $r$ for which $G$ admits an $r$-CNZF. If $G$ has a
bridge then no $r$-CNZF exists for any $r$, and we define
$\phi_c(G)=\infty$.

A circular nowhere-zero {\em modular}-$r$-flow ($r$-MCNZF), is an analogue of an $r$-CNZF,
where the additive group of real numbers is replaced by $\mathbb{R}/r\mathbb{Z}$.
We would like to stress that, given an $r$-MCNZF $f$, the direction of an edge $e$ can be always reversed and $f$ transformed into another $r$-MCNZF, where $f(e) \in \mathbb{R}/r\mathbb{Z}$ is replaced by $-f(e) \in \mathbb{R}/r\mathbb{Z}$.

The following result is well-known and implicitly proved also in Tutte's original work on integer flows \cite{Tu54}.

\begin{proposition}[\cite{Tu54}]\label{mod}
The existence of a circular nowhere-zero $r$-flow in a graph $G$
is equivalent to that of an $r$-MCNZF.
\end{proposition}

The outstanding $5$-Flow Conjecture is equivalent to the statement
that the circular flow number of no bridgeless graph is greater than
$5$. In \cite{EMT15}, the authors present some general methods for
constructing graphs (in particular snarks) with circular flow number
at least $5$. By a direct application of the main results in
\cite{EMT15}, as we have mentioned in the Introduction, one can deduce
that all (few) known snarks with excessive index $5$ have circular
flow number at least $5$. In other words, if a snark is ``critical''
with respect to Berge-Fulkerson's Conjecture, then it seems to be
critical also for the $5$-flow Conjecture. The converse is false, as
shown by the snark $G$ of order $28$, found by M\'{a}\v{c}ajov\'{a}
and Raspaud~\cite{MaRa06}: it has $\phi_C(G)=5$ and
$\chi'_e(G)=4$. 

In this section, we furnish a further element in the direction of the previous observation, by proving that also all treelike snarks have circular flow number at least $5$ (cf. Theorem \ref{main2}).

In order to prove the main theorem of this section, we need to briefly
recall some notions and results introduced in \cite{EMT15}: we will
not present them in the most general setting, but just as needed for
the purpose of this paper. For a general presentation, we refer the
interested reader to the original paper.

First of all, as a direct consequence of Proposition~\ref{mod} we have
the following:
\begin{proposition}[\cite{EMT15}]\label{modular}
For any graph $G$, $\phi_c(G)<5$ if and only if there exists an
$5$-MCNZF $f$ in $G$, such that $f:E\rightarrow(1,4)$.
\end{proposition}

The notion of a $2$-pole is crucial in this setting. A {\em $2$-pole}
$G_{u,v}$ consists of a graph $G$ and two of its vertices, $u$ and
$v$. The vertices $u$ and $v$ are the {\em terminals} of $G_{u,v}$.
The {\em open $5$-capacity} $\mathit{CP}_5(G_{u,v})$ of $G_{u,v}$ is
a subset of $\mathbb{R}/5\mathbb{Z}$, defined by adding to $G$ a new
edge $e$ joining $u$ to $v$, and setting
\begin{equation*}
  \mathit{CP}_5(G_{u,v})=\{f(e)\,|\,f \mbox{ is a modular flow in }G\cup e \mbox{ and } f:\,E(G)\rightarrow (1,4)\}.
\end{equation*}
The following properties hold (see \cite{EMT15}):
\begin{enumerate}[(i)]
 \item The open $5$-capacity of a single edge $[u,v]$ is the open interval $(1,4)$.
 \item The open $5$-capacity of $P^-_{u,v}$ (where $P^-$ is the Petersen graph minus an edge $uv$) is the interval $(4,1)$ in $\mathbb{R}/5\mathbb{Z}$.
\end{enumerate}

\begin{figure}[h]
	\centering
		\includegraphics[height=2.5cm]{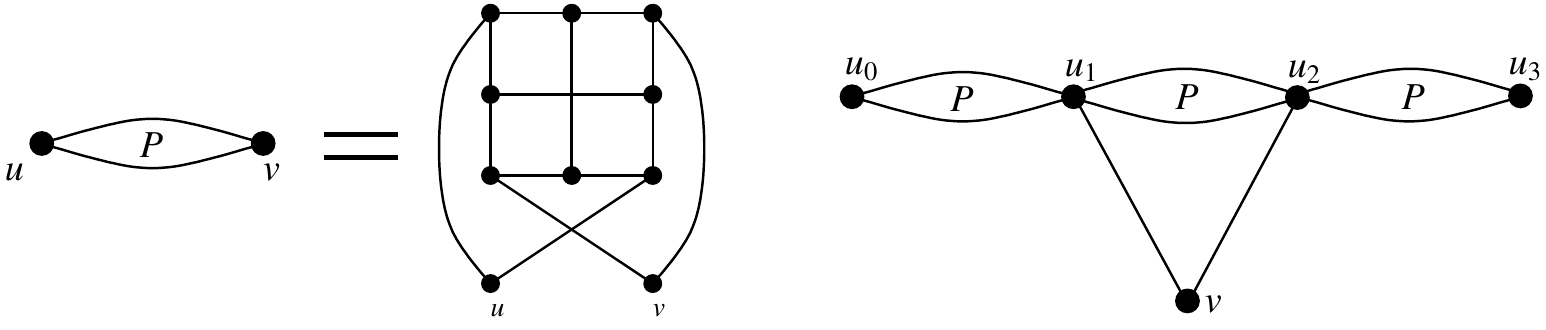}
		\caption{The $2$-pole $P^-_{u,v}$ and the subgraph described in Lemma \ref{flow5}.}\label{2pole}		
\end{figure}

Now we describe a forbidden configuration for a graph whose circular
flow number is less than $5$.

\begin{lemma}\label{flow5}
Let $(u_0,u_1,u_2,u_3)$ be a path in a graph $G$, along vertices of degree $3$ such that there exists a vertex $v$ adjacent both to $u_1$ and $u_2$. Let $G'$ be the graph obtained from $G$ by substituting each of the three edges $[u_i,u_{i+1}]$ for $i=0,1,2$ with a $2$-pole $P^-_{u_i,u_{i+1}}$. Then $\phi_c(G')\geq 5$.
\end{lemma}

\begin{proof}
For the sake of a contradiction, assume $\phi_c(G')<5$. By Proposition \ref{modular}, $G'$ admits a $5$-MCNZF $f'$, such that $f':E\rightarrow(1,4)$. The flow $f'$ of $G'$ induces a flow $f$ on the original graph $G$ with values in $(1,4)$ for all edges but $[u_i,u_{i+1}]$ which have values in $(4,1)$ (since the corresponding $2$-pole has open $5$-capacity $(4,1)$ as previously remarked). Conversely, any such flow of $G$ corresponds to a flow of $G'$ with values in $(1,4)$. Hence, we have to prove that $G$ cannot have a flow as the one described above.
Assume  $(u_0,u_1,u_2,u_3)$ is a directed path in $G$ from $u_0$ to $u_3$. If this is not the case, we can reverse some edge of the path for obtaining a new valid flow of $G$. For the same reason, we can assume $[v,u_1],[v,u_2]$ are directed from $v$ and the third edge in $v$, say $e$, is directed towards $v$ (recall that both $(1,4)$ and $(4,1)$ are symmetric subsets of $\mathbb{R}/5\mathbb{Z}$). Since the flow values of $[v,u_1],[v,u_2]$ belong to $(1,4)$, we cannot have values of $f$ on two consecutive edges of the path $(u_0,u_1,u_2,u_3)$ in the same unit interval: hence, the values of $f$ along the path $(u_0,u_1,u_2,u_3)$ are alternating between the two intervals $(4,0)$ and $(0,1)$. Furthermore, the values of $f$ on $[v,u_1]$ and $[v,u_2]$ are one in $(1,2)$ and the other in $(3,4)$. Finally, the value of $f$ on $e$ is the sum of the values on $[v,u_1]$ and $[v,u_2]$ and so it is in $(4,1)$, that is a contradiction since $e$ has open $5$-capacity $(1,4)$.
\end{proof}

Note that the graph $G'$ constructed in Lemma \ref{flow5} is not cubic since some vertices have degree more than $3$. More precisely, all vertices $u_i$ have degree $5$ in $G'$.
However, it is well known that the expansion of a vertex $x$ to a
subgraph $X$ (see Figure \ref{expansion}) does not decrease the
circular flow number of a graph, since each flow in the expansion can
be naturally reduced to a flow of the original graph. Thus, by
performing a series of expansions, we can transform $G'$ into a cubic
graph without decreasing the circular flow number.

\begin{figure}[h]
	\centering
		\includegraphics[height=3.5cm]{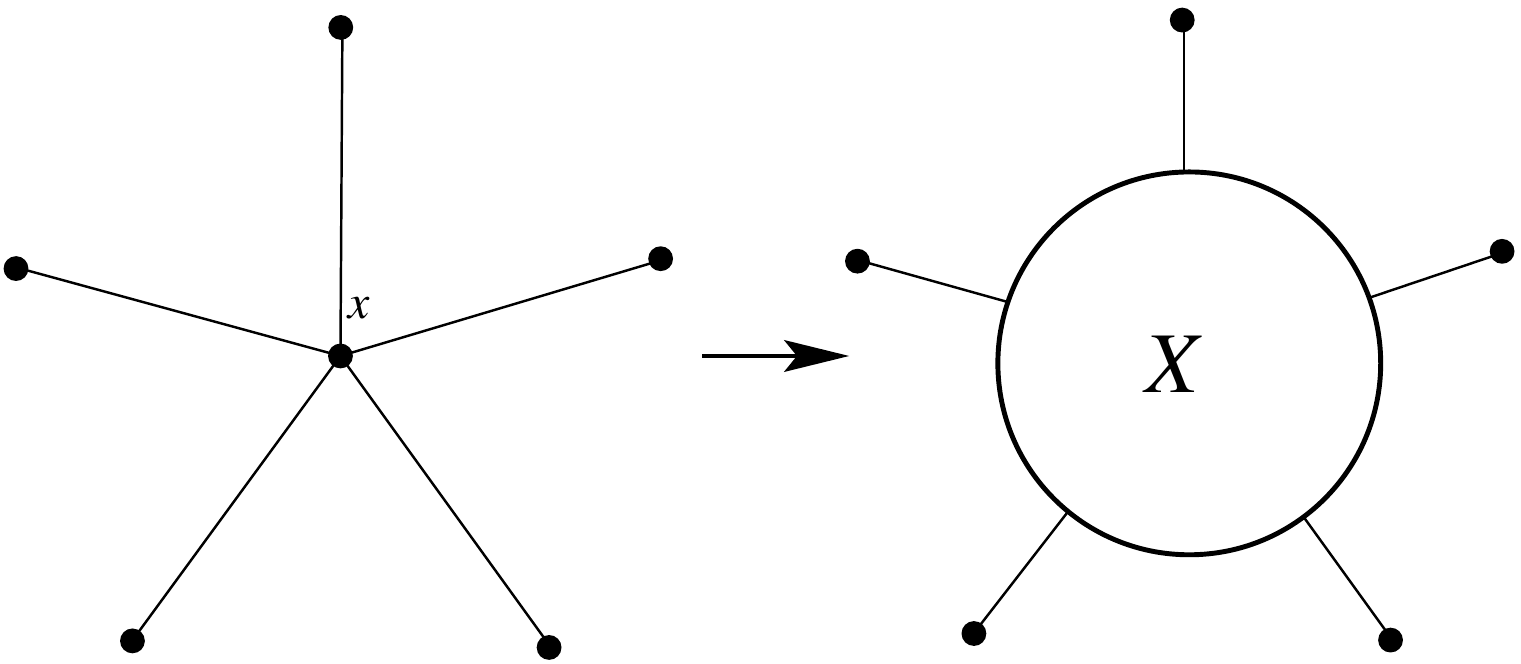}
		\caption{Expansion of a vertex $x$ to an arbitrary subgraph $X$.}\label{expansion}		
\end{figure}

By choosing the starting graph and the expansions in a suitable way,
we obtain the following result for treelike snarks:
\begin{theorem}\label{main2}
Treelike snarks have circular flow number at least $5$.
\end{theorem}
\begin{proof}
  Let $G(T_0,C_0)$ be a treelike snark; consider the corresponding
  cubic Halin graph $H_0$ consisting of a tree $T_0$ and a circuit
  $C_0$ on its leaves. In $H_0$, substitute each edge of $C_0$ with a
  $2$-pole $P^-_{u,v}$. The resulting graph $H'_0$ contains the
  configuration described in Lemma \ref{flow5}, and therefore its
  circular flow number is greater than or equal to $5$. Expanding the
  terminals of every $2$-pole as depicted in Figure \ref{expansion2},
  we obtain $G(T_0,C_0)$. As argued above, vertex expansions do not
  decrease the circular flow number, so the theorem follows.\end{proof}

\begin{figure}[h]
	\centering
		\includegraphics[height=2cm]{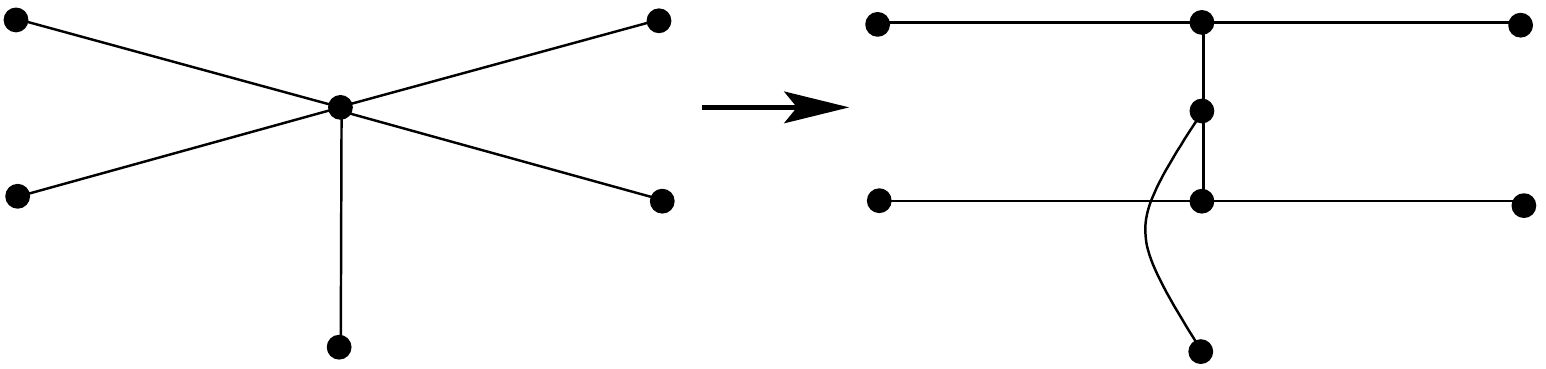}
		\caption{Vertex expansion to obtain treelike snarks.}\label{expansion2}		
\end{figure}

\section{Cycle double covers}\label{cdctls}

In this section, we investigate the properties of treelike snarks with
respect to cycle double covers. Recall that the notion of a
$(1,2)$-cover was defined in Section~\ref{prelim}. A
\emph{$2$-packing} of a cubic graph $G$ a set of joins such that each
edge of $G$ belongs to at most two of the joins.

Hou, Lai and Zhang~\cite{HLZ} have recently proved the following
equivalent formulation of the 5-CDC Conjecture.
\begin{theorem}[\cite{HLZ}, Theorem 3.3] \label{4joins_5CDC}
Let $G$ be a cubic graph. The following statements are equivalent:
\begin{itemize}
\item $G$ has a $5$-cycle double cover,
\item $G$ has a $(1,2)$-cover by $4$ joins.
\end{itemize}
\end{theorem}

As a consequence of Theorem~\ref{4joins_5CDC}, every cubic graph with
excessive index $4$ admits a $5$-cycle double cover, as already
directly proved by Steffen in \cite{S14}. Hence, looking for a
possible counterexample for the $5$-Cycle Double Cover Conjecture, one
should look into the class of snarks with excessive index at least
$5$.

In what follows, we introduce a new general sufficient condition for a cubic graph to admit a $5$-cycle double cover (Theorem \ref{3colconjoin_5CDC}) and we show how it can be easily used to prove the existence of a $5$-cycle double cover for every treelike snark (Theorem \ref{treesnarks_5CDC}).

If $C$ is a circuit in a cubic graph $G$, we let $G_C$ denote the
multigraph obtained by successively contracting each edge of $C$ to a
vertex.

\begin{lemma}\label{ext_join}
Let $G$ be a cubic graph and let $C$ be a circuit of $G$. Then any
join $J'$ of $G_C$ can be extended to a join $J$ of $G$. Moreover, the
previous extension can be performed in two distinct ways.
\end{lemma}
\begin{proof}
Let the vertices of $C$ be denoted by $v_0,\dots,v_t$ in order. For
$i\in\Setx{0,\dots,t}$, let $w_i$ be the unique neighbour of $v_i$
such that $v_iw_i$ is not an edge of $C$. Furthermore, let $v$ denote
the vertex of $G_C$ corresponding to the contracted circuit $C$.

Assume that $J'$ is a given join of $G_C$. Let $I$ be the set of all
$i$ such that $0\leq i\leq t$ and the edge $vw_i$ of $G_C$ does not
belong to $J'$. Since $J'$ is a join of $G'$, $|I|$ is even, say
$|I|=2r$. Let us write $I=\Setx{i_0,\dots,i_{2r-1}}$, where
$i_0<\ldots<i_{2r-1}$. We extend $J'$ to a join $J$ of $G$ as follows:
\begin{itemize}
\item an edge of $G$ not incident with $C$ belongs to $J$ if and only
  if the corresponding edge of $G_C$ belongs to $J'$,
\item an edge $[v_i,w_i]$ of $G$ belongs to $J$ if and only if the corresponding edge $[v,w_i]$ belongs to $J'$,
\item an edge $[v_i,v_{i+1}]$ (indices taken modulo $t$) of $C$
  belongs to $J$ if and only if the relation $i_{2\ell}\leq i < i_{2\ell+1}$
  holds for some non-negative integer $\ell$.
\end{itemize}

To prove the last assertion of the lemma, we can obtain a different
join from $J$ by taking the symmetric difference with $C$ (that is,
removing from $J$ all of its edges contained in $C$, and adding to $J$
all edges of $C$ not in $J$).
\end{proof}

Recall that an edge of a graph is \emph{pendant} if it is incident
with a vertex of degree $1$. Let $J$ be a join of a graph $G$ and
$c:E(J)\to\{1,2,3\}$ be a proper $3$-edge-colouring of $J$.  The
colouring $c$ is said to be {\em congruent} if $|S|\equiv|S \cap
c^{-1}(i)|\pmod 2$ for any set of pendant edges of $J$ forming an
edge-cut of $G$ and for $i=1,2,3$.

Roughly speaking, a $3$-edge colouring of a join is congruent if it
satisfies the condition of the Lemma~\ref{paritylemma} on every set of
pendant edges incident with a circuit in $E(G) \setminus J$.

\begin{lemma}\label{3colconjoin_012cover}
  A cubic graph $G$ admits a connected join with a congruent
  $3$-edge-colouring if and only if it admits a $2$-packing of $4$
  joins, at least one of which is connected.
\end{lemma}
\begin{proof}
  Assume that a join $J_4$ is connected and admits a congruent
  $3$-edge-colouring $c$. In order to prove our assertion, we need to
  construct further joins $J_1,J_2,J_3$ of $G$ such that each edge is
  contained in at most two of $J_1,\dots,J_4$. We choose $J_1,J_2$ and
  $J_3$ using Lemma \ref{ext_join} in such a way that $J_i \cap J_4 =
  \{e \in J_4:\,c(e)=i\}$ for each $i\in\Setx{1,2,3}$. We need to
  modify $\Setx{J_1,\dots,J_4}$ in order to obtain a $2$-packing.
  Clearly, every edge of $J_4$ is covered exactly twice by the four
  joins. The complement of $J_4$ is a family of disjoint circuits. Let
  $C$ be a circuit in $E(G) \setminus J_4$ and note that either all
  edges of $C$ are covered an even number of times or (all of them) an
  odd number of times by $\Setx{J_1,J_2,J_3}$. In the latter case, if
  there is an edge of $C$ covered three times, then we shift the
  selection of one join, say $J_1$, on $C$ according to the second
  part of Lemma \ref{ext_join}. In this way, every edge of $C$ is
  covered $0$ or $2$ times as desired. By repeating the same process
  on each circuit of $E(G) \setminus J_4$, we obtain a join $J'_1$
  such that $\Setx{J'_1,J_2,J_3,J_4}$ is a $2$-packing and $J_4$
  connected.

  Conversely, suppose $\{J_1,J_2,J_3,J_4\}$ a $2$-packing of joins
  with $J_4$ connected. We have to prove that $J_4$ has a congruent
  $3$-edge-colouring. If an edge is incident to a vertex of degree 3 of
  $J_4$, then it is covered exactly twice: it belongs to $J_4$ and to
  exactly one of the joins $J_1$,$J_2$ and $J_3$. But, since $J_4$ is
  connected, every edge of $J_4$ has at least one end-vertex which has
  degree $3$ in $J_4$. Hence, $J_1 \cap J_4$, $J_2 \cap J_4$ and $J_3
  \cap J_4$ induce a $3$-edge-colouring of $J_4$. Moreover, the
  $3$-edge-colouring is congruent since Lemma \ref{paritylemma} holds
  for the intersection of any join with an edge-cut of $G$.
\end{proof}

\begin{theorem}\label{3colconjoin_5CDC}
Let $G$ be a cubic graph. If $G$ admits a connected join with a congruent $3$-edge-colouring, then it has a $5$-cycle double cover.
\end{theorem}
\begin{proof}
  By Lemma \ref{3colconjoin_012cover}, we have that $G$ admits a
  $2$-packing $\cal J$ of $4$ joins with at least one of them
  connected, say $J_4$.  Since $J_4$ is connected, for every edge $e$
  uncovered in $\cal J$, there exist a circuit $C_e$ such that $e \in
  C_e$ and all other edges of $C_e$ belong to $J_4$. Construct $J'_4$
  as the symmetric difference of $J_4$ and all cycles $C_e$. Every
  edge of $J_4$ is still covered in $J'_4$ at least once, since it
  belongs to one of the joins $J_1$,$J_2$ and $J_3$, while every edge
  not covered in $\cal J$ belongs to $J'_4$, hence the set
  $\{J_1,J_2,J_3,J'_4\}$ is a $(1,2)$-covering by $4$ joins and by
  Theorem \ref{4joins_5CDC}, $G$ has a $5$-cycle double cover.
\end{proof}

\begin{corollary}\label{connected_complement}
Let $G$ be a cubic graph. If $G$ admits a $3$-edge-colourable connected join $J$ with a connected complement in $G$, then $G$ has a $5$-cycle double cover.
\end{corollary}
\begin{proof}
The complement of $J$ is connected if and only if it is a unique circuit of $G$. The unique edge-cut of $G$ with pendant edges of $J$ is the entire set of all pendant edges, hence Lemma \ref{paritylemma} holds. It follows that any $3$-edge-colouring of $J$ is congruent and the assertion follows by previous theorem.
\end{proof}

Now we prove that every treelike snark has a $5$-cycle double cover.

Let $H$ be the $5$-sunlet, that is the graph on $10$ vertices obtained
by attaching $5$ pendant edges to a $5$-circuit (see Figure
\ref{3colouring}).  In order to prove Theorem \ref{treesnarks_5CDC}, we
will make use of the following lemma.
\begin{lemma}\label{5sunset}
  Let $H$ be the $5$-sunlet. If we prescribe the colours of any two
  non-consecutive pendant edges of $H$, we can complete this
  prescription to a proper $3$-edge-colouring of $H$.
\end{lemma}
\begin{proof}
The desired colouring can be easily obtained with a suitable permutation of the colouring $c'$ in Figure \ref{3colouring}.
\end{proof}

\begin{figure}[h]
	\centering
		\includegraphics[width=5cm]{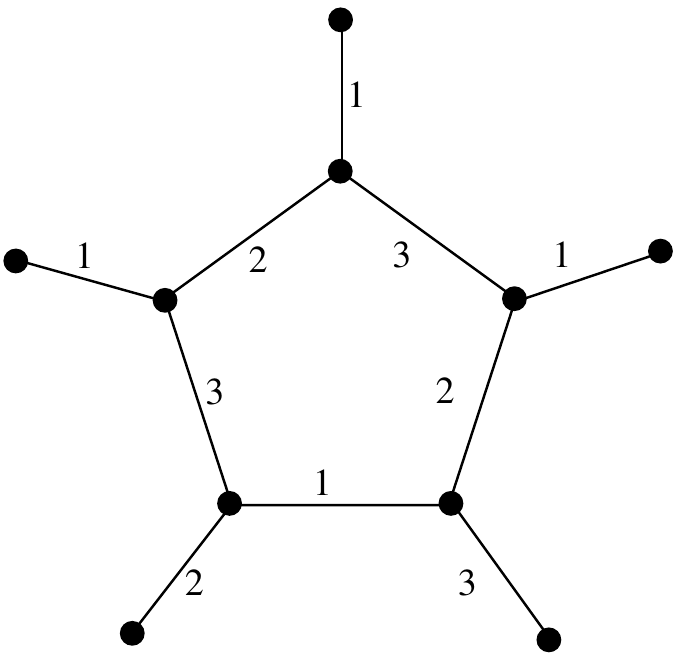}
		\caption{A proper $3$-edge-colouring $c'$ of the $5$-sunlet.}\label{3colouring}		
\end{figure}

\begin{theorem}\label{treesnarks_5CDC}
Treelike snarks admit a 5-cycle double cover.
\end{theorem}
\begin{proof}
  By Corollary \ref{connected_complement} and Theorem
  \ref{3colconjoin_5CDC}, it is sufficient to prove that every
  treelike snark admits a connected $3$-edge-colourable join which is
  the complement of a circuit. Let $G=G(T_0,C_0)$ be a treelike
  snark. Recall that for each leaf $\ell$ of $T_0$, there is a
  Petersen fragment $F^\ell_0$; suppose that its loose half-edges are
  denoted by $(a^\ell_1,\dots,a^\ell_5)$. In addition, let $P^\ell$
  denote the unique path of length $3$ connecting the endvertices of
  $a^\ell_1$ and $a^\ell_5$.

  Let $C$ be the union of the paths $P^\ell$ (with $\ell$ ranging over
  the leaves of $T_0$), together with the edges connecting their
  endvertices (i.e., the edges containing a loose half-edge $a^\ell_1$
  or $a^\ell_5$ of some Petersen fragment $F^\ell_0$). By the
  construction of $G$, $D$ is a circuit. (See Figure
  \ref{fragment_decomp} where the edges of $C$ are shown dashed.)

  Let $J$ be the join obtained as complement of $C$ in $G$. The join
  $J$ is connected because there is a path from any vertex of $J$ to a
  vertex of the tree $T_0$. We now prove that $J$ is
  $3$-edge-colourable. Let $X$ be the set of all edges in each Petersen
  fragment $F^\ell_0$ that are not contained in $P^\ell$, do not
  contain any loose half-edge of $F^\ell_0$, and are not incident with
  $a^\ell_1$ nor $a^\ell_2$. The edges in $X$ are shown bold in Figure
  \ref{fragment_decomp}. By inspecting the figure, we find that $G-X$
  is a tree of maximum degree $3$; let us denote it by $T$, Clearly,
  $T$ admits a proper $3$-edge-colouring $c$. By Lemma \ref{5sunset},
  whatever are the colours of the two edges of $T$ incident with the
  unique $5$-circuit of $X$ inside any fragment $F^\ell_0$, we can
  extend $c$ to a proper $3$-edge-colouring on $F^\ell_0$. Eventually,
  we obtain a proper $3$-edge-colouring of the join $J$.
\end{proof}

\begin{figure}[h]
	\centering
		\includegraphics[width=7cm]{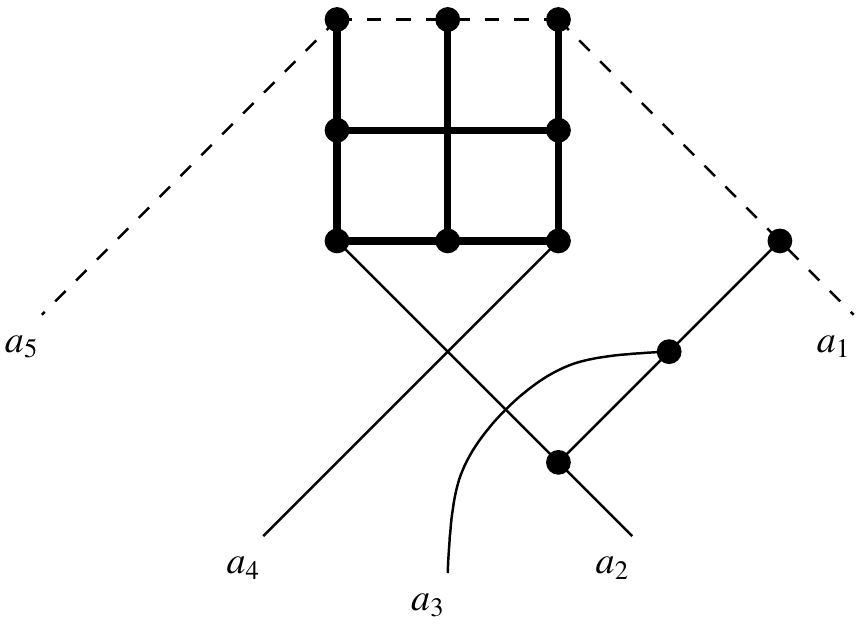}
		\caption{An illustration to the proof of
                  Theorem~\ref{treesnarks_5CDC}. Edges of the circuit
                  are shown dashed, edges of $X$ are shown
                  bold.}\label{fragment_decomp}
\end{figure}

We conclude this section by pointing out a reformulation of
Theorem~\ref{3colconjoin_5CDC} in terms of nowhere-zero flows:
\begin{corollary}
  Let $G$ be a cubic graph admitting a connected join $J$. Let $G'$ be
  the graph obtained from $G$ by contracting each circuit of the
  complement of $J$ to a vertex. If $G'$ admits a nowhere-zero
  $4$-flow, then $G$ admits a $5$-cycle double cover. 
\end{corollary}
\begin{proof}
  It is well known that if $G'$ admits a nowhere-zero $4$-flow, then
  its edges can be coloured with three colours such that the edges of
  each colour constitute a join. The corresponding colouring of the
  edges of $J$ is a congruent 3-edge-colouring, so the hypothesis of
  Theorem~\ref{3colconjoin_5CDC} is satisfied and the claim follows.
\end{proof}

\section{Pattern sets}\label{appendix}

The notions of a pattern and the pattern set of a fragment were
introduced in Section~\ref{sec:patterns}. The pattern sets in this
section were determined by computer enumeration. We do not list
patterns obtained from the listed ones by one or more of the following
operations (cf. the discussion above Proposition~\ref{tls-snark}):
\begin{itemize}
\item interchanging the first two subsets of $\alphabet$ in the pattern,
\item interchanging the last two subsets of $\alphabet$ in the pattern,
\item permuting the elements of $\alphabet$.
\end{itemize}

\subsection{The pattern set of the Petersen fragment}\label{app-f0}
\begin{minipage}[t]{14cm}

\noindent The pattern set of $F_0$ (42 patterns):

\begin{minipage}[t]{3cm}
\small
A A AB AC AD  \\
A A AB C D  \\
A AB A AC AD \\
A AB A BC BD \\
A AB AC A AD  \\
A AB AC B BD  \\
A AB AC C CD  \\
A B AB AB CD  \\
A B AB AC BD  \\
A B AC A D  \\
A B AC AB BD  \\
A B C A AD  \\
A B C C CD  \\
A B CD A A  \\
\end{minipage}
\hfill
\noindent
\begin{minipage}[t]{3cm}
\small
A B CD AB AB  \\
A B CD AC AC  \\
A B CD C C \\
A B CD CD CD \\
A BC A AB BD  \\
A BC B AB AD  \\
A BC B BC CD  \\
A BC BD A AB  \\
A BC BD BC C  \\
A BC BD BD D  \\
A BC D A A  \\
A BC D AB AB  \\
A BC D AD AD  \\
A BC D B B  \\
\end{minipage}
\hfill
\noindent
\begin{minipage}[t]{4cm}
\small
A BC D BC BC \\
A BC D BD BD  \\
A BC D D D  \\
AB AB AB AC AD\\
AB AC AB AB AD  \\
AB AC AB BC CD  \\
AB AC AD A A  \\
AB AC AD AB AB  \\
AB AC AD AD AD \\
AB AC AD B B  \\
AB AC AD BC BC  \\
AB AC AD BD BD  \\
AB AC AD D D  \\
AB CD AC AB BC  \\
\end{minipage}
\end{minipage}

\noindent \rule{14cm}{.4pt}

\subsection{The pattern sets of sums of Petersen fragments}\label{app-f0sums}
\begin{minipage}[t]{14cm}

\noindent \noindent The pattern set of $F_0+F_0$ (18 patterns):

\begin{minipage}[t]{3cm}
\small
A AB C AB BD  \\
A AB C AC CD  \\
A B AC AB BD \\
A B AC AC CD \\
A B C A AD  \\
A B C C CD \\
\end{minipage}
\hfill
\noindent
\begin{minipage}[t]{3cm}
\small
A BC A A D  \\
A BC A AB BD  \\
A BC B AB AD  \\
A BC B B D  \\
A BC B BC CD  \\
A BC D A A \\
\end{minipage}
\hfill
\noindent
\begin{minipage}[t]{4cm}
\small
A BC D AB AB \\
A BC D AD AD  \\
A BC D B B  \\
A BC D BC BC \\
A BC D BD BD \\
A BC D D D \\
\end{minipage}
\end{minipage}

\noindent \rule{14cm}{.4pt}

\noindent \begin{minipage}[h]{14cm}
\noindent The pattern set of $F_0+(F_0+F_0)$ (9 patterns): \\

\noindent \begin{minipage}[h]{5cm}
\small
A AB C AB BD \\
A AB C AC CD  \\
A B AC AB BD \\
A B AC AC CD  \\
A B C A AD  \\
A B C C CD  \\
A BC A AB BD  \\
A BC B AB AD  \\
A BC B BC CD \\
\end{minipage}
\end{minipage}

\

\

\noindent \rule{14cm}{.4pt}

\noindent \begin{minipage}[h]{14cm}
\noindent The pattern set of $(F_0+F_0)+F_0$ (25 patterns):\\

\noindent \begin{minipage}[t]{3cm}
\small
A AB A AC AD \\
A AB A BC BD \\
A AB C AB BD\\
A AB C AC CD  \\
A B AC AB BD  \\
A B AC AC CD  \\
A B C A AD  \\
A B C C CD  \\
A BC A A D  \\
\end{minipage}
\hfill
\noindent
\begin{minipage}[t]{3cm}
\small
A BC A AB BD  \\
A BC AD AB B \\
A BC AD AD D  \\
A BC B AB AD  \\
A BC B B D  \\
A BC B BC CD  \\
A BC BD A AB  \\
A BC BD BC C  \\
\end{minipage}
\hfill
\noindent
\begin{minipage}[t]{4cm}
\small
A BC BD BD D \\
A BC D A A \\
A BC D AB AB \\
A BC D AD AD  \\
A BC D B B  \\
A BC D BC BC  \\
A BC D BD BD  \\
A BC D D D  \\
\end{minipage}
\end{minipage}

\noindent \rule{14cm}{.4pt}

\noindent \begin{minipage}[h]{14cm}
\noindent The pattern set of $(F_0+F_0)+(F_0+F_0)$ (10 patterns):\\

\noindent \begin{minipage}[h]{5cm}
\small
A B AC AB BD \\
A B AC AC CD\\
A BC A AB BD\\
A BC AD AB B\\
A BC AD AD D\\
A BC B AB AD\\
A BC B BC CD\\
A BC BD A AB\\
A BC BD BC C\\
A BC BD BD D\\
\end{minipage}
\end{minipage}


\end{document}